\documentclass[11pt]{amsart}
\usepackage{amsfonts}
\usepackage{graphicx}
\usepackage{amsmath}
\usepackage{amsthm}
\usepackage{amssymb}
\usepackage{amscd}
\usepackage{color}
\usepackage{url}

\usepackage{graphicx}
\usepackage{microtype}
\usepackage{enumitem}
\usepackage{pinlabel}
\usepackage[top=1.3in,bottom=1.4in,left=1.3in,right=1.3in]{geometry}

\theoremstyle{definition}

\newtheorem{theorem}{Theorem}[section]
\newtheorem{lemma}[theorem]{Lemma}
\newtheorem{corollary}[theorem]{Corollary}

\newtheorem*{proposition*}{Proposition}

\theoremstyle{definition}

\newcommand{\R}{\mathbb{R}}
\newcommand{\Z}{\mathbb{Z}}

\newcommand{\id}{id}

\DeclareMathOperator{\Fix}{\mathsf{Fix}}

\DeclareMathOperator{\Diff}{\mathsf{Diff}}
\DeclareMathOperator{\Homeo}{\mathsf{Homeo}}

\DeclareMathOperator{\Osupp}{\mathsf{Osupp}}

\newcounter{notes}
\renewcommand{\paragraph}[1]{\medskip 
\noindent \textbf{#1}}

\title{There are no exotic actions of diffeomorphism groups on 1-manifolds}

\address{\newline Department of Mathematics   \newline California Institute of Technology   \newline Pasadena, CA 91125,  USA \newline
Department of Mathematics \newline  Cornell University  \newline Ithaca, NY 14853, USA}
\email{chenlei@caltech.edu, k.mann@cornell.edu}
\author{Lei Chen and Kathryn Mann}
\begin{document}

\maketitle

\begin{abstract}
Let $M$ be a manifold, $N$ a 1-dimensional manifold.  Assuming $r \neq \dim(M)+1$, we show that any nontrivial homomorphism $\rho: \Diff^r_c(M)\to \Homeo(N)$ has a standard form:  necessarily $M$ is $1$-dimensional, and there are countably many embeddings $\phi_i: M\to N$ with disjoint images such that the action of $\rho$ is conjugate (via the product of the $\phi_i$) to the diagonal action of $\Diff^r_c(M)$ on $M \times M \times ...$ on $\bigcup_i \phi_i(M)$, and trivial elsewhere.  This solves a conjecture of Matsumoto.  We also show that the groups $\Diff^r_c(M)$ have no countable index subgroups.
 
\end{abstract}
\section{Introduction}
Let $\Diff^r_c(M)$ denote the identity component of the group of compactly supported $C^r$ diffeomorphisms of a manifold $M$.  In this paper, we prove the following statement.  

\begin{theorem}\label{main}
Let $M$ be a connected manifold, and suppose that $\rho: \Diff^r_c(M) \to \Homeo(N)$ is a homomorphism, where $N = S^1$ or $N = \R$, $r \neq \dim(M)+1$.   Then $\dim(M) = 1$ and there are countably many disjoint embeddings $\phi_i : M \to N$ such that $\rho(g)|_{\phi_i(M)} = \phi_i g \phi_i^{-1}$ and $N - \bigcup_i \phi_i(M)$ is globally fixed by the action.  
\end{theorem}

This proves \cite[Conjecture 1.3]{Matsumoto} and
generalizes works of Mann \cite{Mann_ETDS}, Militon \cite{Militon} and Matsumoto \cite{Matsumoto}, but with an independent proof.   Matsumoto's work  \cite{Matsumoto} proves an analogous result when the target is $\Diff^1(N)$ using rigidity theorems of \cite{BMNR} for solvable affine subgroups of $\Diff^1(\R)$.  This generalized \cite{Mann_ETDS}, which proved the result for homomorphisms to $\Diff^2(N)$ using Kopell's lemma.  Militon \cite{Militon} studies homomorphisms where the source is the group of {\em homeomorphisms} of $M$.   Our proof here is comparatively short, and is self-contained modulo the standard but difficult result that $\Diff^r_c(M)$, for $r \neq \dim(M) +1$ is a simple group.   Whether this holds for $r = \dim(M)+1$ is an open question; this is responsible for our restrictions on dimension in the statement. 

Theorem \ref{main} is already known in the case where $\rho$ is assumed to be continuous; it is a consequence of the orbit classification theorem of \cite{ChenMann}, and was likely known to others before.  In the case where the target is the group of smooth diffeomorphisms of $N$, this also follows from work of Hurtado \cite{Hurtado} who proves additionally that any such homomorphism is necessarily (weakly) continuous.  
Here we make no assumptions on continuity, however, our proof suggests that diffeomorphism groups exhibit ``automatic continuity"--like properties.  Specifically, we show the following {\em small index property}.  

\begin{theorem}[The small index property of $\Diff^r_0(M)$]\label{smallindex}
If $r \neq \dim(M) + 1$, then $\Diff_c^r(M)$ has no proper countable index subgroup.   Equivalently, $\Diff_c^r(M)$ has no nontrivial homomorphism to the permutation group $S_\infty$.
\end{theorem}

This is in stark contrast with the case for finite dimensional Lie groups, where we have the following.   
\begin{theorem}[Thomas \cite{Thomas} and Kallman \cite{Kallman}]\label{TK}
There is an injective homomorphism $\mathrm{SL}_n(\R) \to S_\infty$.
\end{theorem}

Thus, one consequence of Theorem \ref{smallindex} and \ref{TK} is that there is no nontrivial homomorphism from $\Diff_c^r(M)$ into a linear group.   Of course, this is nearly immediate if one considers only continuous homomorphisms, since $\Diff_c^r(M)$ is infinite dimensional, and one may simply cite the invariance of domain theorem.  

If $G$ is a group with a non-open subgroup $H$ of countable index, then the action of $G$ on the coset space $G/H$ gives a discontinuous homomorphism to $S_\infty$.  This is one of very few known general recipes for producing discontinuous group homomorphisms (see \cite{Rosendal_BSL}), so gives some (weak) evidence that $\Diff_c^r(M)$ might have the automatic continuity property already known to hold for $\Homeo(M)$ by \cite{MannCont}, or for homomorphisms between groups of smooth diffeomorphisms as in \cite{Hurtado}.  

Theorem \ref{main} also gives new examples of left orderable groups that do not act on the line.  It is a well known fact that any {\em countable} group with a left-invariant total order admits a faithful homomorphism to $\Homeo_+(\R)$.   
For $r>0$, the groups $\Diff_c^r(\R^n)$ for $r>0$ are known to be left-orderable: the Thurston stability theorem \cite{Thurston} implies that they are locally indicable (any finitely generated subgroup surjects to $\Z$), which implies that they are left-orderable by the Burns-Hale theorem (\cite{BH}, see also \cite[Corollary 2]{Rolfsen}).   Thus, we have the following.  

\begin{corollary}
For $r >0$, the group $\Diff_c^r(\R^n)$ is left-orderable but has no faithful action on the line or the circle.
\end{corollary}

\vskip 0.3cm 

The proof of Theorem \ref{smallindex} uses the idea from the first step of the proof of automatic continuity for homeomorphism groups of  \cite{MannCont}, following Rosendal \cite{R}. 
This result is then used to prove Theorem \ref{main} by constraining the supports and fixed sets of elements for the action on $N$.  We are then able to use this information to build a map from $M$ to $N$.  


\section{Proof of the small index property}
In this section, we prove Theorem \ref{smallindex}. The proof follows the local version of the arguments in \cite{MannCont} and \cite{R} for automatic continuity of $\Homeo_0(M)$.  

%

\begin{proof}
Let $M$ be a manifold and $r \neq \dim(M) +1$.  Let $G = \Diff^r_c(M)$, and for an open subset $U\subset M$, denote by $G_U$ the subgroup of $\Diff^r_c(M)$ consisting of elements whose support is compactly contained in $U$.   Thus, $G_U \cong \Diff^r_c(U).$
Suppose for contradiction that $H \subset G$ is a countable index subgroup.    We will show in step 1 that there is some ball $U$ in $M$ such that $G_U \subset H$.  After this, we will show that $H$ acts transitively on $M$, thus every $x \in M$ is contained in some open set $U_x$ such that $G_{U_x} \subset H$.  The {\em fragmentation property} gives that $\Diff^r_c(M)$ is generated by the union of such sets $G_{U_x}$ (this is true for any collection of sets $U_x$ which form an open cover of $M$, see \cite[Ch.1]{Banyaga}), so this is sufficient to prove $H =G$.  

\paragraph{Step 1.} Let $g_1H, g_2H, \ldots$ denote the left cosets of $H$.  Let $B \subset M$ be any ball, and take a sequence of disjoint balls $B_i \subset M$, with diameter tending to 0 and such that the sequence $B_i$ Hausdorff converges to a point.  
  
We first claim that there exists some $i$ such that every element $f\in G_{B_i}$ sufficiently close to the identity agrees with the restriction of an element $w_f \in g_iH$ to $B_i$.  Furthermore, we will have that $w_f$ is supported on $B$.   To prove this, let $U_i$ be an identity neighborhood of $G_{B_i}$, chosen small enough so that for any sequence of diffeomorphisms $f_i \in U_i$,  the infinite composition $\prod_i f_i$ is an element of $G$.  Equivalently, $f_i$ must converge to the identity in $G$ fast enough.    Supposing our claim is not true, we can find $h_i \in H$ such that the restriction of $g_i h_i$ to $B_i$ does not lie in $U_i$.  But $\prod_i (g_i h_i) \in G$ so lies in some coset $g_jH$ and restricts to $g_jh_j$ on each $B_j$, a contradiction.    We have in fact shown something stronger, for if $f \in G_{B_j}$, then $w_{\id}^{-1} w_f \in H$ and restricts to $f$ on $B_j$, so this shows that every element in $U_j$ agrees with the restriction of an element of $H$ to $B_j$.  Since $U_j$ generates $G_{B_j}$, we conclude that every element of $G_{B_j}$ agrees with the restriction of an element of $H$ to $B_j$.  

Now we use a commutator trick.  Apply the same argument as above using $B_j$ in place of $B$.  We find a smaller ball $B' \subset B_j$ such that every element $f \in G_{B'}$ agrees with the restriction to $B'$ of an element $v_f \in H$, and $v_f$ is supported on $B_j$. 
Since $\Diff_c(B')$ is perfect \cite{Anderson, Mather1, Mather2, Thurston_foliations}, any element $f \in \Diff_c(B')$ may be written as a product of commutators $f = \prod[a_i, b_i]$.  But $[a_i, b_i] = [v_{a_i}, w_{b_i}]$ since the supports of $v_{a_i}$ and $w_{b_i}$ intersect only in $B'$, and so $f = \prod [v_{a_i}, w_{b_i}] \in H$. 
This ends the proof of the first step.  

\paragraph{Step 2: transitivity.} 
To prove transitivity, let  $B'$ be the ball from step 1, and let $x \in B'$.  Suppose $y \in M$ is some point {\em not} in the orbit of $x$.  Let $f_t$ be a flow such that $f_t(y) \in B'$ for all $t \in (1,2)$.   Such a flow can be defined to have support on a neighborhood of a path from $x$ to $y$.  Since $B'$ lies in the orbit of $x$ under $H$, we have that $f_t \notin H$ for $t \in (1,2)$.   We know that  $H \cap \{ f_t : t \in \R\}$ is a countable index subgroup of $\{ f_t : t \in \R\} \cong \R$.  In particular, it must intersect every open interval of $\R$, this gives the desired contradiction.  
\end{proof}

As an immediate consequence, we can conclude that any fixed point free action of such a group on the line or circle is minimal. 

\begin{corollary} \label{cor:countable_index}
With the same restrictions on $r$ as above, if $\Diff^r_c(M)$ acts on $\R$ or $S^1$ without global fixed points, then there are no invariant open sets.  In particular the action has a dense orbit.  
\end{corollary}

\begin{proof}
Suppose the action has an invariant open set.  Then $\Diff^r_c(M)$ permutes the (countably many) connected components of $U$.  The stabilizer of an interval is a countable index subgroup, so by Theorem \ref{smallindex}, the permutation action is trivial.  Thus each interval is fixed and their endpoints are global fixed points.  
\end{proof}


\section{Proof of Theorem \ref{main}}
For the proof Theorem \ref{main}, we set the following notation.  As in the previous section we fix some $r \neq \dim(M)+1$ and when $U \subset M$ is an open set we denote by $G_U$ the set of elements of $\Diff^r_c(M)$ supported on $U$.   Also, $G^U \subset \Diff^r_c(M)$ denotes the set of elements that pointwise fix $U$.  The {\em open support} of a homoemorphism $g$ is the set $\Osupp(g) := N - \Fix(g)$; as is standard, the {\em support} of $g$ is defined to be the closure of $\Osupp(g)$.

\begin{proof}
We will assume the action on $N$ has no global fixed points, since if the action does have fixed points, then $N - \Fix(\rho)$ is a union of open intervals, each with a fixed-point free action of $\Diff^r_c(M)$, so it suffices to understand such actions.  In this case, we will show that there is a single homeomorphism $\phi:M \to N$ such that the action on $N$ is induced by conjugation by $\phi$.    

\begin{lemma} \label{lem:disjoint} 
For any action, if $U \cap V = \emptyset$, then $\Osupp(\rho(G_U)) \cap \Osupp(\rho(G_V)) = \emptyset$. 
\end{lemma} 

\begin{proof} 
Since $G_U$ and $G_V$ commute, $\rho(G_V)$ preserves $\Osupp(\rho(G_U))$, permuting its connected components.  By Theorem \ref{smallindex}, this action is trivial.   Let $I$ be a connected component of $\Osupp(\rho(G_U))$.  Suppose $\rho(G_V)$ 
acts nontrivially on $I$.   Since $G_V$ is simple group, its action on $I$ is faithful.  Since $G_V$ is {\em not} abelian, H\"{o}lder's theorem implies that some nontrivial $\rho(g) \in \rho(G_V)$ acts with a fixed point.   But then $\rho(G_U)$ permutes the connected components of $I - \Osupp(\rho(g))$, and this permutation action is trivial.  Thus, $\rho(G_U)$ has a fixed point in $I$, contradicting that $I \subset \Osupp(\rho(G_U))$.
\end{proof}

Also, if $\bar U \cap \bar V = \emptyset$ then $G^U$ and $G^V$ generate $\Diff^r_c(M)$, so our assumption that there are no global fixed points for the aciton implies that $\Fix(\rho(G^U)) \cap \Fix(\rho(G^V) = \emptyset$ as well.
 
Our next goal is to define a map from $M$ to $N$. 
For each $x \in M$ pick a neighborhood basis $U_n$ of $x$ so $\bigcap_n U_n = \{x\}$.  
Let $S_x = \bigcap_n \Osupp(\rho(G_{U_n}))$ and let $T_x = \bigcap_n \Fix(\rho(G^{U_n}))$.  Note that this is independent of the choice of neighborhood basis.  

\begin{lemma} \label{lem:disjoint2}
If $x \neq y$, then $S_x \cap S_y = \emptyset$ and $T_x \cap T_y = \emptyset$.  Also, $S_x$ and $T_x$ have empty interior. 
\end{lemma}

\begin{proof} 
The first assertion follows immediately from Lemma \ref{lem:disjoint} and the second because $T_x \cap T_y$ would be globally fixed by $\rho$ by our observation above.  
Furthermore, if $g(x) = y$, then it follows from the definition that 
\[ \rho(g)S_x = \bigcap_n \rho(g) \Osupp(\rho(G_{U_n})) = \bigcap_n  \Osupp(\rho(g) \rho(G_{U_n}) \rho(g)^{-1})  = \bigcap_n  \Osupp(\rho(G_{g(U_n)}).
 \]
Thus, $\rho(g)S_x = S_y$.  
Similarly we have $T_y = \rho(g)T_x$.  Thus, if some $S_x$ has nonempty interior, disjointness of $S_x$ and $S_y$ would give an uncountable family of disjoint open sets in $N$, a contradiction.   The same applies to the sets $T_x$.  
\end{proof}  

We next prove these sets, though defined differently, are in fact the same.  
\begin{lemma}
$S_x = T_x$
\end{lemma}

\begin{proof}  Fix $x$ and let $U_n$ be a neighborhood basis of $x$.  Since $G^{U_n} \subset G_{N-\overline{U_{n+1}}}$ and since $N-\overline{U_n}$ and $U_n$ are disjoint, by Lemma \ref{lem:disjoint}, we have that 
\[
\Osupp(\rho(G_{U_{n+1}})=N-\Fix(\rho(G_{U_{n+1}})) \subset \Fix(\rho(G^{U_n}).\]

Thus $S_x \subset T_x$.  For the reverse inclusion, suppose $z \in T_x - S_x$.  Then $z \notin \Osupp(\rho(G_{U_n}))$ for some $n$; i.e.,  $z\in \Fix(\rho(G_{U_n}))
$.  Also $z\in \Fix(\rho(G^{U_{n+1}}))$ by the definition of $T_x$. But $G_{U_n}$ and $G^{U_{n+1}}$ together generate $\Diff^r_c(M)$ (this again is the {\em fragmentation property}), so this implies that $z$ is a global fixed point.  
\end{proof}  
\begin{lemma}
$S_x$ is nonempty. 
\end{lemma}
\begin{proof}
If the action is on $S^1$, this follows immediately since $S_x = T_x$ is the intersection of nested, nonempty closed sets.  If $N = \R$, the same is true provided that $\Fix(\rho(G^{U_n}))$, (or equivalently $\Osupp(\rho(G_{U_n})$), does not leave every compact set as $n \to \infty$.  
Supposing for contradiction that this is true, this means that for each $x \in M$ and compact $K \subset \R$ there is a neighborhood $U(x)$ such that $\Osupp(\rho(G_{U(x)})) \cap K = \emptyset$.  Fix any compact subset $A \subset M$.  Then finitely many of these neighborhoods $U(x)$, for $n \in A$ cover $A$.  But then the union of these finitely many subgroups $G_{U(x)}$ generate $\Diff_c(A) \subset \Diff_c(M)$, hence $K \cap \Osupp(\Diff_c(A)) = \emptyset$.  Since every element of $\Diff_c(M)$ lies in $\Diff_c(A)$ for some compact subset $A$ of $M$, we conclude that  $K \cap \Osupp(\Diff_c(M)) = \emptyset$ contradicting that $\rho$ had no global fixed points in $N$.  
\end{proof}

\paragraph{Construction of $\phi$.}
To finish the proof, we wish to show that $S_x$ is a singleton, and the assignment $\phi: x \mapsto S_x$ is a homeomorphism conjugating $\rho$ with the standard action of $\Diff(M)$ on $M$.   We will actually show first that $x \mapsto S_x$ is a local homeomorphism,  use this to {\em conclude} that $S_x$ is discrete, and proceed from there.  

Let $I = (a,b)$ be a connected component of $N - S_x$, chosen so that $a \neq -\infty$ if $N = \R$.  If $N = S^1$ and $S_x$ is a singleton, it is possible that both ``endpoints'' of this interval agree.  For simplicity, we treat the case where $a \neq b$, the case $a = b$ on the circle can be handled with exactly the same strategy, and in fact the argument simplifies quite a bit since $S_x$ is already a singleton.  
Let $U$ be a neighborhood of $x$ small enough so that $b \notin \Osupp(\rho(G_U))$; this is possible by definition of $S_x$.   Then for each $g \in G_U$, $g(a)< b$.  

\paragraph{Step 1: definition of $\phi$ locally}

Note that $a$ is not accumulated from the right in $S_x$. For any $n$, denote by $O_n$ the connected component of $\Osupp(\rho(G_{U_n}))$ that contains $a$. Since $a$ cannot be accumulated to the right by points of $S_x$ (i.e. on at least the right side it looks like an isolated point of $S_x$), there exists $k \in \mathbb{N}$ such that $a$ is the right most point of $S_x\cap O_k$.   We claim that, for $y\in U_k$, the set $S_y\cap O_k$ also has a rightmost point, in which case we define $\phi(y)$ to be this rightmost point.

To see such a rightmost point exists, take $g\in G_{U_k}$ with $g(x)=y$. Then $\rho(g)(S_x)=S_y$. Since $\rho(g)$ fixes endpoints of $O_k$ by definition, we know that $\rho(g)(a)\in S_y$, which is also the rightmost point of $S_y\cap O_k$.  Thus, $\phi$ is a well defined function on $O_k$ with image contained in $U_k$.    
An equivalent definition of $\phi$ is that $\phi(y):=\rho(g)(a)$, where $g$ is any diffeomorphism in $G_{U_k}$ such that $g(x) = y$.  Our argument above shows this is independent of choice of $g$.  

\paragraph{Step 2: local continuity of $\phi$ on $U_k$}
We first show that $\phi$ is continuous at $x$.  Suppose $x_n \to x$ is a convergent sequence.  Passing to a subsequence and reindexing if needed, we may assume that $x_n \in U_n$.  Then we may take $g \in G_{U_n}$ so that $g(x) = x_n$, so $\phi(x_n) = \rho(g)(x)$.  Since the sequence of connected components of $\Osupp(\rho(G_{U_n}))$ containing $x$ converges to $x$, we get that $\phi(x_n) \to x$.   

Now we show continuity on $U_k$.  {Now take any point $x' \in U_k$, and a sequence $x'_n \to x$ in $U_k$.  There exists $g \in G_{U_k}$ such that $g(x)=x'$ so $g^{-1}(x'_n)$ is a sequence converging to $x$.  It follows from continuity at $x$ that $\phi(g^{-1}(x'_n))$ converges to $\phi(x)$.  
By definition, $\rho(g)\phi(g^{-1}(x'_n)) = \phi(x'_n)$, so we conclude that $\phi(x'_n)$ converges to $\phi(x)$.  }

Note also that $\phi$ is injective by Lemma \ref{lem:disjoint2}.  Thus, by invariance of domain, we conclude that $M$ is one-dimensional so equal to $\R$ or $S^1$, and $\phi$ gives a homeomorphism from $U_k$ onto an open neighborhood $A$ of $a$ in $N$.  In particular, this shows that $a$ is an isolated point of $S_x$.  

\paragraph{Step 3: extension of $\phi$ globally}

The last step is to show that $\phi$ extends naturally to a globally defined homeomorphism $M \to N$.
Note first that the orbit of $A$ under $\rho(G)$ is an open, $\rho(G)$-invariant set, so by Corollary \ref{cor:countable_index}, $\rho(G)(A) = N$.

This topological transitivity implies that, for all $x$, every point of $S_x$ is an isolated point, and the map $S_x \mapsto x$ is a covering map, and it is equivariant with respect to $\rho$ since $\rho(g)(S_x) = S_{g(x)}$.  In the case $M = S^1$, it follows that $N = S^1$ and the cover is degree one, as can be seen by considering the subgroup of rotations.  In the case $M = \R$ we immediately have that $N = \R$.  

\end{proof}

    	\bibliography{citing}{}
	\bibliographystyle{plain}
\end{document}